\documentclass[11pt]{amsart}
\usepackage[utf8]{inputenc}
\usepackage[english]{babel}
\usepackage{multirow}
\usepackage{anysize}
\usepackage{color}
\usepackage{multicol}
\usepackage{tikz-cd}

\usepackage{enumitem}
\usepackage{amsfonts}
\usepackage{amsmath}
\usepackage{amssymb}
\usepackage{amsthm}

\usepackage[usenames,dvipsnames]{pstricks}
 \usepackage{pstricks-add}
 \usepackage{epsfig}
\usepackage{mathrsfs}

\newtheorem{Th}{Theorem}[section]
\newtheorem{D}[Th]{Definition}
\newtheorem{Ob}[Th]{Remark}
\newtheorem{Cor}[Th]{Corollary}

\newtheorem{Pro}[Th]{Proposition}

\newtheorem{Le}[Th]{Lemma}

\title[Controllability of the Lorentzian rolling on ${\mathbb R}^{n,1}$]{Controllability of the rolling system of a Lorentzian manifold on ${\mathbb R}^{n,1}$}

\author[Abraham Bobadilla Osses and Mauricio Godoy Molina]{Abraham Bobadilla Osses$^{1,2}$ and Mauricio Godoy Molina$^2$}

\address{$^1$ Facultad de Ingenier\'ia, Universidad Aut\'onoma de Chile – Sede Temuco, Av. Alemania 1090 Temuco, Chile.}
\email{abraham.bobadilla@uautonoma.cl}

\address{$^2$ Departamento de Matem\'atica y Estad\'istica, Universidad de La Frontera, Av. Francisco Salazar 01145 Temuco, Chile.}
\email{abraham.bobadilla@ufrontera.cl, mauricio.godoy@ufrontera.cl}



\thanks{This research is partially supported by Fondecyt \#1181084}

\begin{document}
\maketitle
\begin{abstract}
In this paper, we study the mechanical system associated with rolling a Lorentzian manifold $(M,g)$ of dimension $n+1\geq2$ on flat Lorentzian space $\widehat{M}={\mathbb R}^{n,1}$, without slipping or twisting. Using previous results, it is known that there exists a distribution $\mathcal{D}_R$ of rank $(n+1)$ defined on the configuration space $Q(M,\widehat{M})$ of the rolling system, encoding the no-slip and no-twist conditions. Our objective is to study the problem of complete controllability of the control system associated with $\mathcal{D}_R$. The key lies in examining the holonomy group of the distribution $\mathcal{D}_R$ and, following the approach of \cite{ChKok}, establishing that the rolling problem is completely controllable if and only if the holonomy group of $(M,g)$ equals $SO_0(n,1)$.
\end{abstract}

\section{Introduction}

Motions of systems with nonholonomic constraints can be found in classical problems in mechanics and are a source of very interesting problems in mathematics, especially in relation to differential geometry and control theory, see for example \cite{ABl,J,Mon,Sol} for modern accounts of these interactions. For several decades nonholonomic constraints have been studied from the point of view of sub-Riemannian geometry, that is when the constraints define a non-integrable subbundle of the tangent bundle of a smooth manifold on which a metric is defined, for instance in \cite{Bro}. In recent years, several authors have studied a generalization of the previous framework, considering that the subbdundle is equipped with a pseudo-Riemannian metric -- {\it i.e.}, nondegenerate but not positive definite. The problems that appear in this context fall within the so-called sub-pseudo-Riemannian geometry, see \cite{ChaMar,GKM,Groch2,Groch1,Groch3,AnnIr,AnnIr2}. Most classical results in pseudo-Riemannian geometry and some applications to physics can be found in \cite{On}.

The mechanical system of a pair of $n-$dimensional Riemannian manifolds rolling on each other without slipping and twisting provides a myriad of examples of nonholonomic systems, since these constraints are not integrable for ``most'' pairs of manifolds, see for example \cite{ChKok,ChKo1,ChKokM,MGrIrFa,Nom}. The rolling problem is known for being technically challenging and for displaying interesting geometric behaviors. The case in which two pseudo-Riemannian manifolds roll on each other has only recently been addressed in \cite{IrFa}, and to better understand the geometry of this mechanical system, one needs certain tools from sub-pseudo-Riemannian geometry. Specific examples of rolling pseudo-Riemannian hyperquadrics can be found in \cite{AnnFa,MSL}.

We will study the problem of rolling, without slipping or twisting, a Lorentzian manifold $(M,g)$ on the flat Lorentzian space $({\mathbb R}^{n,1},\hat{g})$, both of dimension $n+1\geq3$. One of the most emblematic problems related to the rolling problem is giving conditions for the complete controllability of the system, that is, determining whether any given configuration of the manifolds can be rolled to any other, as mentioned before without slipping or twisting. The aim of this article is to provide a relation between the Lorentzian holonomy of $M$ and the controllability of the system, inspired by \cite{ChKok}, but taking care of several technical difficulties that appear due to the non-positivity of the metric. The study of the pseudo-Riemannian holonomy is a matter of current research, due to the fact that the group of isometries of the tangent space is no longer compact, see \cite{Anton,AThomas,Leistner}.


The structure of the paper is as follows. In Section \ref{sec:construction} we present the basic definitions to construct the so-called {\it rolling lift} and the {\it rolling distribution}. This point of view, which differs significantly from the definition given in \cite{IrFa} is much more convenient for our purposes. Afterwards in Section \ref{sec:action} we look carefully at the action of the group of Lorentzian affine rigid transformations $SE_0(n,1)$ on the rolling problem (or, more precisely, the connected component of the identity). This is the most delicate part of the argument. In Section \ref{sec:control}, we apply the previous analysis to the formulation of the controllability of the Lorentzian rolling system in terms of the Lorentzian holonomy of the rolling manifold $M$. Finally, in Section \ref{sec:ex}, we present a few simple examples and interesting consequences.


\section{Construction of the pseudo-Riemannian rolling system}\label{sec:construction}

\subsection{Notation} 

Let $M$ be an $m-$dimensional connected pseudo-Riemannian manifold of index $\nu$ and metric tensor $\langle\cdot,\cdot\rangle_J$. Recall that $\nu$ corresponds to the maximum dimension of a subspace of $T_xM$, $x\in M$, where $\langle\cdot,\cdot\rangle_J$ is negative definite. Given a point $x\in M$ and a tangent vector $v\in T_xM$, it is said that $v$ is
\begin{itemize}
\item spacelike if $\langle v,v\rangle_J>0$ or $v=0$, 
\item timelike if $\langle v,v\rangle_J<0$, and 
\item lightlike if $\langle v,v\rangle_J=0$ and $v\neq0$.
\end{itemize}
If the metric tensor has index $\nu=1$ and $m\geq2$, then $M$ is called a Lorentzian manifold. Later in this article we will require the manifold to be Lorentzian, but most of the concepts in this section hold for arbitrary pseudo-Riemannian manifolds.

Given a smooth distribution $\mathcal{D}$ on $M$, we say that an absolutely continuous (a.c.) curve $\gamma:I=[0,1]\to M$ is $\mathcal{D}-$admissible if $\gamma$ is tangent to $\mathcal{D}$ almost everywhere, that is, if for almost all $t\in I$ it holds that $\dot{\gamma}(t)\in\mathcal{D}|_{\gamma(t)}$. For $x_0\in M$, the endpoints of all the $\mathcal{D}-$admissible curves of $M$ starting at $x_0$ form the set called $\mathcal{D}-$orbit through $x_0$ and denoted $\mathcal{O}_{\mathcal{D}}(x_0)$. More precisely,
\begin{equation}\label{eq:orbit}
\mathcal{O}_{\mathcal{D}}(x_0)=\{\gamma(1)|\,\gamma:I\to M,\text{ $\mathcal{D}-$admissible},\,\gamma(0)=x_0\}.
\end{equation}

The parallel transport of a tangent vector $v\in T_{x}M$ along an a.c. path $\gamma:[a,b]\to M$ from $\gamma(a)=x$ to $\gamma(b)$, associated to the Levi-Civita connection $\nabla$, is denoted as $P_a^b(\gamma)v$. The parallel transport operator
$$P_a^b(\gamma):T_{\gamma(a)}M\to T_{\gamma(b)}M$$
is a linear isometry and it satisfies 
\[
P_b^a(\gamma)=P_a^b(\gamma^{-1})=P_a^b(\gamma)^{-1},
\]
where $\gamma^{-1}$ denotes the curve $\gamma$ parameterized backwards.

For $x\in M$, we use $\Omega_x(M)$ to denote the set of all piecewise a.c. loops $I\to M$ based at $x$. This set becomes a monoid using the usual concatenation of paths. We define the holonomy group of $\nabla$ at $x_0$ as the subgroup $H_{x_0}$ of $GL(T_{x_0}M)$ given by
$$H_{x_0}=\left\{P_0^1(\gamma)|\gamma\in\Omega_{x_0}(M)\right\}.$$
This group has the structure of a Lie group, see \cite{Nom}. It is a standard result that, since $M$ is connected, the holonomy groups at different points are isomorphic.


We use $\operatorname{Iso}(M,g)$ to denote the group of isometries of a pseudo-Riemannian manifold $(M,g)$. The isometries respect parallel translation in the sense that for any a.c. curve $\gamma:I\to M$ and $F\in\operatorname{Iso}(M,g)$ one has
$$F_*|_{\gamma(t)}\circ P_0^t(\gamma)=P_0^t(F\circ\gamma)\circ F_*|_{\gamma(0)},$$
see \cite{Sak}.

Recall that the development of an a.c. curve $\gamma:I\to M$, $\gamma(0)=x$, is a curve $\hat\gamma\colon I\to T_xM$, $\hat\gamma(0)=0$, given by the endpoints of the parallel transport $P_t^0(\gamma)\colon T_{\gamma(t)}M\to T_xM$ applied to the vector field $\dot\gamma$ along $\gamma$, see \cite[Chapter 3, Proposition 4.1]{KoNo}. The following formula follows directly from the definition of development and the Riemannian analogue can be found, for example in \cite[Proposition 2.1]{ChKok}.

\begin{Pro}
The development $\hat\gamma\colon I\to T_xM$ of $\gamma:I\to M$, $\gamma(0)=x$ is given by
$$\hat\gamma(t)=\int_0^tP_s^0(\gamma)\dot{\gamma}(s)\,ds\in T_xM,\quad t\in I.$$
This defines a map $\Lambda_x:\gamma\mapsto\hat\gamma$ from the space of a.c. curves $I\to M$ starting at $x$ to the space of a.c. curves $I\to T_xM$ starting at $0$.
\end{Pro}

\begin{D}
A pseudo-Riemannian manifold is called:
\begin{enumerate}
\item {\em geodesically complete} if every maximal geodesic is defined on the entire real line,

\item {\em development complete} if the development map $\Lambda_x$ at $x\in M$, is bijective.
\end{enumerate}
\end{D}

For the main results of this paper the hypothesis of development completeness is crucial.

\begin{Ob}
In the Riemannian setting, it is well-known that geodesic completeness is equivalent to development completeness (see \cite[Chapter 4, Proposition 4.1]{KoNo}). Results of this nature for pseudo-Riemannian manifolds are a matter of intense research to this day even for Lorentzian manifolds, for example see \cite{LM}. In fact, unlike in the Riemannian case, there are simple examples of compact Lorentzian manifolds that are not geodesically complete (the classic example is the Clifton-Pohl torus). The injectivity of the development map appears as a hypothesis in certain important results in Lorentzian geometry, see \cite{Kami}.
\end{Ob}

\subsection{State space and distribution}

Given $V$ and $W$ two oriented $m-$dimensional real vector spaces endowed with non-degenerate scalar products of the same index $\nu$. We denote by $G(V,W)$ the group of all orientation preserving linear isometries between $V$ and $W$. For further details, see \cite{IrFa}.

For any pair $M$ and $\widehat{M}$ of connected and oriented $m-$dimensional pseudo-Riemannian manifolds both of index $\nu$, we introduce the space $Q=Q(M,\widehat{M})$ of all relative positions in which $M$ can be tangent to $\widehat{M}$
$$Q=\left\{q\in G(T_{x}M,T_{\widehat{x}}\widehat{M})|x\in M,\widehat{x}\in\widehat{M}\right\}.$$
It is a manifold with the structure of a $G_\nu(m)-$fiber bundle over $M\times\widehat{M}$, where $G_\nu(m)<O_\nu(m)$ is the subgroup of the pseudo-orthogonal group corresponding to the orientation. The dimension of $Q$ is $2m+(m(m-1))/2=m(m+3)/2$. It will be referred to as the state space of the rolling.

Let $x_0\in M$ and $\widehat{x}_0\in\widehat{M}$. Given $R_0\in G\left(T_{x_0}M,T_{\widehat{x}_0}\widehat{M}\right)$ and smooth curves $\gamma\colon I\to M$ and $\widehat{\gamma}\colon I\to\widehat{M}$ such that $\gamma(0)=x_0$ and $\widehat{\gamma}(0)=\widehat{x}_0$, we consider the curve of isometries
$$P_0^t(\widehat{\gamma})\circ R_0\circ P_t^0(\gamma)\colon T_{\gamma(t)}M\to T_{\widehat{\gamma}(t)}\widehat{M},\quad t\in I$$
It can be easily checked that this curve lies in $Q$. Let $\lambda_M$ and $\lambda_{\widehat{M}}$ be the respective orientations of $T_xM$ and $T_{\widehat{x}}\widehat{M}$, then
\begin{align*}
  P_0^t(\gamma)\lambda_M(x_0) &= \lambda_M(x(t)) \\
  P_t^0(\gamma)\lambda_M(x(t)) &= \lambda_M(x_0) \\
  R_0\circ P_t^0(\gamma)\lambda_M(x(t)) &= R_0\lambda_M(x_0)=\lambda_{\widehat{M}}(\widehat{x}_0) \\
  P_0^t(\widehat{\gamma})\circ R_0\circ P_t^0(\gamma)\lambda_M(x(t)) &= P_0^t(\widehat{\gamma})\lambda_{\widehat{M}}(\widehat{x}_0) \\
   &= \lambda_{\widehat{M}}(\widehat{x}(t))
\end{align*}

Following the references \cite{ChKok,ChKo1}, we can define the rolling motion using the so-called rolling lift. 

\begin{D}\label{def:lift}
Let $q=(x,\widehat{x};A)\in Q$ and $X\in T_xM$. The rolling lift of $X$ at $q$ is the vector $\mathscr{L}_R(X)|_q\in T_qQ$ given by
 $$\mathscr{L}_R(X)|_q=\left.\frac{d}{dt}\right|_0(P_0^t(\widehat{\gamma})\circ A\circ P_t^0(\gamma)),$$
where $\gamma\colon I\to M$ and $\widehat{\gamma}\colon I\to\widehat{M}$ are smooth curves that satisfy 
\begin{equation}\label{eq:req}
\gamma(0)=x_0,\quad\widehat{\gamma}(0)=\widehat{x}_0,\quad\dot{\gamma}(0)=X,\quad\mbox{and}\quad\dot{\widehat{\gamma}}(0)=AX.
\end{equation}
\end{D}

From the chain rule it is clear that Definition \ref{def:lift} is independent of the curves $\gamma$ and $\widehat{\gamma}$, as long as they satisfy the requirements \eqref{eq:req}. Let us briefly remark a matter of notation. Even though in general the configuration space $Q$ does not have the structure of a product manifold, it is standard to write its elements as $ q=(x,\widehat{x};A) \in Q$, where $A:T_xM\to T_{\widehat{x}}\widehat{M}$ is an isometry.

This induces a map $\mathscr{L}_R:\mathfrak{X}(M)\to \mathfrak{X}(Q)$ as follows. For a vector field $X\in \mathfrak{X}(M)$ we define the lifted vector field $\mathscr{L}_R(X)\in\mathfrak{X}(Q)$ by
$$
\begin{array}{ccccl}
\mathscr{L}_R(X): & Q & \to & TQ \\
  & q & \mapsto & \mathscr{L}_R(X)|_q
\end{array}
$$

The rolling lift map $\mathscr{L}_R$ allows one to construct a distribution on $Q$ of rank $m$ as follows.

\begin{D}
The rolling distribution $\mathcal{D}_R$ on $Q$ is the $m-$dimensional smooth distribution defined by
$$\mathcal{D}_R|_q=\mathscr{L}_R(T_xM)|_q,\quad q=(x,\widehat{x};A) \in Q.$$
\end{D}

It is a crucial observation that the map $\pi_{Q,M}=\operatorname{pr}_1\circ\pi_Q:Q\to M$ is a fiber bundle projection, namely, if $F=(e_i)_{i=1}^m$ is a local oriented orthonormal frame of $M$, defined on an open set $U$, the local trivialization of $\pi_{Q,M}$ induced for $F$ as
$$\begin{array}{ccccl}
\tau_{F}: & \pi^{-1}_{Q,M}(U) & \to & U\times F_{OON}(\widehat{M}) \\
 & q=(x,\widehat{x};A) & \mapsto & (x,(Ae_i)_{i=1}^m)
\end{array}
$$
is a diffeomorphism, where $F_{OON}(\widehat{M})$ is the bundle of all oriented orthonormal frames on $\widehat{M}$.

The following result can be found in \cite{ChKo1} in the Riemannian setting. For rolling pseudo-Riemannian manifolds it is important to assume that at least one of the manifolds is development complete.

\begin{Pro}
\begin{enumerate}
\item[i)] For any $q_0=(x_0,\hat{x}_0;A_0)\in Q$ and a.c. $\gamma:I\to M$, such that $\gamma(0)=x_0$, there exists a unique a.c. $q:[0,a]\to Q$, $q(t)=(\gamma(t),\hat{\gamma}(t);A(t))$, with $0\leq a\leq1$, which is tangent to $\mathcal{D}_R$ a.e. and $q(0)=q_0$. We denote this unique curve $q$ by
$$t\mapsto q_{\mathcal{D}_R}(\gamma,q_0)(t)=(\gamma(t),\hat{\gamma}_{\mathcal{D}_R}(\gamma,q_0)(t);A_{\mathcal{D}_R}(\gamma,q_0)(t)).$$
and refer to it as the rolling curve along $\gamma$ with initial position $q_0$. In the case that $\widehat{M}$ is development complete one has $a=1$.

Conversely, any a.c. curve $q:I\to Q$, which is a.e. tangent to $\mathcal{D}_R$, is a rolling curve along $\gamma:=\pi_{Q,M}\circ q$, i.e. has the form $q_{\mathcal{D}_R}(\gamma,q(0))$.

\item[ii)] For any $q_0=(x_0,\hat{x}_0;A_0)\in Q$ and a.c. curve $\gamma$ starting from $x_0$, the corresponding rolling curve is given by
\begin{equation}\label{eq:rollcurve}
q_{\mathcal{D}_R}(\gamma,q_0)(t)=\left(\gamma(t),\widehat{\Lambda}_{\hat{x}_0}^{-1}(A_0\circ\Lambda_{x_0}(\gamma))(t);P_0^t(\widehat{\Lambda}_{\hat{x}_0}^{-1}(A_0\circ\Lambda_{x_0}(\gamma)))\circ A_0\circ P_t^0(\gamma)\right),
\end{equation}
where $\Lambda$ and $\widehat\Lambda$ are the development maps on $M$ and $\widehat M$, respectively.
\end{enumerate}
\end{Pro}

Before starting with the more technical aspects of the present paper, let us first clarify that using the rolling system one can show that development completeness of $M$ implies its geodesic completeness.

\begin{Pro}
If the development map $\Lambda_x\colon\gamma\mapsto\hat\gamma$ is bijective, then $M$ is geodesically complete.
\end{Pro}

\begin{proof}
It is well-known that the rolling system defined above satisfies the so-called no twist condition, which states that a vector field along a curve $\gamma\colon I\to M$ is parallel if and only if its image under a rolling is parallel along the image curve. See \cite[Proposition 1]{IrFa}. This translates to the fact that the rolling map and the covariant derivative commute

If $\Lambda_x$ is bijective, then $\Lambda_x^{-1}\ell$ is a well-defined curve on $M$ for an arbitrary straight line $\ell\colon{\mathbb R}\to T_xM$ parameterized with constant speed. Since $\ell$ is a geodesic in $T_xM$, it follows from the no twist condition that $\Lambda_x^{-1}\ell$ is a geodesic in $M$ defined on the entire real line.
\end{proof}



\section{Action of $SE_0(n,1)$ on the rolling problem on ${\mathbb R}^{n,1}$}\label{sec:action}

For the rest of the paper we assume that $(M,g)$ is a $(n+1)-$dimensional Lorentzian manifold, we fix a point $x_0\in M$, and take $\widehat{M}=T_{x_0}M\cong\mathbb{R}^{n,1}$.

As was mentioned at the end of the previous section, the fiber bundle $\pi_{Q,M}:Q\to M$ is key to understanding the rolling system for this specific choice of $\widehat{M}$. It is known that for general Riemannian manifolds $M$ and $\widehat{M}$, the projection map $\pi_Q:Q\to M\times\widehat{M}$ is not a principal bundle, except when $n=2$, see \cite[Proposition 3.4]{ChKokArxiv}. That this fiber bundle happens to have an appropriate $SO(2)-$action for surfaces, has been extensively used, for example in \cite{AS,BH,MaBi}.

As was mentioned after Definition \ref{def:lift}, the elements in $Q$ are written as triplets even though in general $Q$ is not a product manifold. Under the assumption that $\widehat{M}=\mathbb{R}^{n,1}$, that abuse of notation becomes even more useful, since the following diagram
\[\begin{tikzcd}
	Q && M\times SE_0(n,1) \\
	\\
	&& M
	\arrow[from=1-1, to=3-3, "\pi_{Q,M}",swap]
	\arrow[from=1-1, to=1-3, "f"]
	\arrow[from=1-3, to=3-3, "\operatorname{pr}_1"]
\end{tikzcd}\]
commutes, where $f(x,\hat{x};A)=(x,(\hat{x},A))$ and $\operatorname{pr}_1$ is the projection onto the first component. In this argument we consider $SE_0(n,1)$ as the semidirect product ${\mathbb{R}}^{n,1}\rtimes SO_0(n,1)$.

Moreover, due to the choice of $\widehat{M}$, it can be seen that $\pi_{Q,M}$ is a principal bundle with a left $SE_0(n,1)-$action $\mu$, given by
\begin{equation}\label{eq:action}
\mu(B,q)=(x,C\hat{x}+\hat{y};CA).
\end{equation}
where $q=(x,\hat{x};A)\in Q$ and $B=(\hat{y},C)\in SE_0(n,1)$. To see that this is indeed the case, let us first show that $\mu(B,q)\in Q$ for all $q\in Q$ and $B\in SE_0(n,1)$. Taking a tangent vector $X\in T_xM$ and considering $\mu(B,q)$ as a linear map from $T_xM\to T_{\hat{z}}\mathbb{R}^{n,1}$, where $\hat{z}=C\hat{x}+\hat{y}$, we have
\[
\|\mu(B,q)X\|_J^2=\langle\mu(B,q)X,\mu(B,q)X\rangle_J=\langle CAX,CAX\rangle_J=\langle AX,AX\rangle_J=\langle X,X\rangle_g.
\]
This shows that $CA$ is an isometry. Additionally, since $SO_0(n,1)$ is the identity component of $SO(n,1)$, it preserves orientation. Thus, $\mu(B,q)$ is indeed an element of $Q$.

Considering that $\mu:SE_0(n,1)\times Q\to Q$ is well defined and smooth, the next step is to study the fibers of $\pi_{Q,M}$. For each $x\in M$ and each $q\in\pi^{-1}_{Q,M}(x)$, with $q=(x,\hat{x};A)$ the map
\begin{equation}\label{eq:action}
\begin{array}{rcl}
\mu_q : SE_0(n,1) &\longrightarrow & \pi^{-1}_{Q,M}(x) \\
(\hat{y},C)&\longmapsto& (x,C\hat{x}+\hat{y};CA)
\end{array}
\end{equation}
is a diffeomorphism. It is clearly smooth, and its inverse $\mu^{-1}_q(x,\hat{z};D)=(\hat{z}-DA^{-1}\hat{x},DA^{-1})$ is also smooth. This proves that $\mu_q$ is indeed an $SE_0(n,1)-$bundle.

It remains to show that $\mu$ is free and that is transitive on each fiber. If 
\begin{equation}\label{eq:free}
\mu(B,q)=(x,C\hat{x}+\hat{y};CA)=q=(x,\hat{x};A),
\end{equation}
then focusing on the third component of equality \eqref{eq:free}, we have
\[
CA=A\implies C=\operatorname{id}.
\]
Similarly, looking at the second component of equality \eqref{eq:free}, it follows that
\[
C\hat{x}+\hat{y}=\hat{x}\implies\hat{x}+\hat{y}=\hat{x}\implies\hat{y}=0.
\]
Therefore $B=(\hat{y},C)=(0,\operatorname{id})$, and thus the action $\mu$ is free. 

For the fiber transitivity of $\mu$, let $q=(x,\hat{x},A)$ and $\overline{q}=(x,\overline{\hat{x}};\overline{A})$ belong to the same fiber of $\pi_{Q,M}$. The goal is to find $B=(\hat{y},C)\in SE_0(n,1)$ such that $\mu(B,q)=\overline{q}$, or more explicitly, we want to solve the equation
\[
(x,C\hat{x}+\hat{y};CA)=(x,\overline{\hat{x}};\overline{A}).
\]
This is clearly equivalent to the system
\begin{align}
C\hat{x}+\hat{y}&=\overline{\hat{x}}\label{eq:trans}\\
CA &= \overline{A}\label{eq:rot}
\end{align}
We can solve equation \eqref{eq:rot} for $C$ to find $C = \overline{A}A^{-1}$. Substituting this into equation \eqref{eq:trans}, we get $\overline{A}A^{-1}\hat{x}+\hat{y} = \overline{\hat{x}}$. Rearranging terms, we have $\hat{y} = \overline{\hat{x}} - \overline{A}A^{-1}\hat{x}$. This gives us a solution $B=(\overline{\hat{x}} - \overline{A}A^{-1}\hat{x},\overline{A}A^{-1})\in SE_0(n,1)$, and thus, the action is transitive over each fiber.

The following result strengthens the previous argument, showing that the action $\mu$ preserves the rolling distribution and its relation to the orbits of the rolling problem.

\begin{Le}\label{lem:distr}
\begin{enumerate}
\item[(i)] The action \eqref{eq:action} preserves the distribution $\mathcal{D}_R$, that is, for any $q\in Q$ y $B\in SE_0(n,1)$, 
\begin{equation*}
(\mu_B)_*\mathcal{D}_R|_q=\mathcal{D}_R|_{\mu(B,q)},
\end{equation*}
where $\mu_B:Q\to Q$ is defined by $\mu_B(q)=\mu(B,q)$.
\item[(ii)] For each $q=(x,\hat{x};A)\in Q$ there exists a unique subgroup $H_q<SE_0(n,1)$ called the holonomy group of $\mathcal{D}_R$, such that
$$\mu(H_q\times{q})=\mathcal{O}_{\mathcal{D}_R}(q)\cap\pi^{-1}_{Q,M}(x),$$
where $\mathcal{O}_{\mathcal{D}_R}(q)$ is the orbit of the rolling distribution passing through $q\in Q$, as in \eqref{eq:orbit}. Moreover, if $q'=(x',\hat{x}';A')\in Q$ is in the same fiber as $q$, then $H_q$ y $H_{q'}$ are conjugate in $SE_0(n,1)$ and every conjugation class of $H_q$ in $SE_0(n,1)$ is of the form $H_{q'}$. This conjugation class will be denoted by $H$. Furthermore $\pi_{\mathcal{O}_{\mathcal{D}_R},M}:\mathcal{O}_{\mathcal{D}_R}(q)\to M$ is an $H$ bundle over $M$.

\end{enumerate}
\end{Le}

\begin{proof}
\begin{enumerate}
\item[(i)] To prove that the action $\mu$ preserves the distribution, let us consider $B=(\hat{y},C)\in SE_0(n,1)$. Let $\gamma\colon I\to M$ be an a.c. curve, an initial configuration $q_0\in Q$, and $\hat{\gamma}_{\mathcal{D}_R}(\gamma,q_0)$ be the corresponding development. First, we will show that the following commuting relation
\[
\mu_B(q_{\mathcal{D}_R}(\gamma,q_0))=q_{\mathcal{D}_R}(\gamma,\mu_B(q_0))
\]
is valid.

From the third coordinate in equation \eqref{eq:rollcurve}
$$P_0^t(\hat{\gamma}_{\mathcal{D}_R}(\gamma,q_0))AP_t^0(\gamma)=A_{\mathcal{D}_R}(\gamma,q_0)(t),$$
we have
$$AP_t^0(\gamma)\dot{\gamma}(t)=P_t^0(\hat{\gamma}_{\mathcal{D}_R}(\gamma,q_0))\dot{\hat{\gamma}}_{\mathcal{D}_R}(\gamma,q_0)(t).$$
Letting $B\in SE_0(n,1)$ act on the previous equality, we see that
$$P_t^0(\hat{\gamma}_{\mathcal{D}_R}(\gamma,\mu_B(q_0)))\dot{\hat{\gamma}}_{\mathcal{D}_R}(\gamma,\mu_B(q_0))=CAP_t^0(\gamma)\dot{\gamma}(t).$$
As a consequence, we have the equalities
\begin{align*}
P_t^0(C\hat{\gamma}_{\mathcal{D}_R}(\gamma,q_0))\frac{d}{dt}(C\hat{\gamma}_{\mathcal{D}_R}(\gamma,q_0)(t))&=CP_t^0(\hat{\gamma}_{\mathcal{D}_R}(\gamma,q_0))\dot{\hat{\gamma}}_{\mathcal{D}_R}(\gamma,q_0)(t)\\
&= CAP_t^0(\gamma)\dot{\gamma}(t)\\
&= P_t^0(\hat{\gamma}_{\mathcal{D}_R}(\gamma,\mu_B(q_0)))\dot{\hat{\gamma}}_{\mathcal{D}_R}(\gamma,\mu_B(q_0))(t)
\end{align*}
due to the uniqueness of solutions of ordinary differential equations, we conclude that $C\hat{\gamma}_{\mathcal{D}_R}(\gamma,q_0)=\hat{\gamma}_{\mathcal{D}_R}(\gamma,\mu_B(q_0))$. Hence,
\begin{align*}
CA_{\mathcal{D}_R}(\gamma,q_0)&=C(P_0^t(\hat{\gamma}_{\mathcal{D}_R}(\gamma,q_0))AP_t^0(\gamma))=P_0^t(C\hat{\gamma}_{\mathcal{D}_R}(\gamma,q_0))CAP_t^0(\gamma)\\
&=P_0^t(\hat{\gamma}_{\mathcal{D}_R}(\gamma,\mu_B(q_0)))CAP_t^0(\gamma)=A_{\mathcal{D}_R}(\gamma,\mu_B(q_0))
\end{align*}
This shows that $\mu_B(q_{\mathcal{D}_R}(\gamma,q_0)(t))=q_{\mathcal{D}_R}(\gamma,\mu_B(q_0))(t)$.

Differentiating with respect to $t$ and evaluating at $t=0$, we have
\begin{align*}
(\mu_B)_*\mathscr{L}_R(\dot{\gamma}(0))|_{q_0}&= (\mu_B)_*\left.\frac{d}{dt}\right|_{t=0}q_{\mathcal{D}_R}(\gamma,q_0)(t)\\
&= \left.\frac{d}{dt}\right|_{t=0}\mu(B,q_{\mathcal{D}_R}(\gamma,q_0))\\
&= \left.\frac{d}{dt}\right|_{t=0}q_{\mathcal{D}_R}(\gamma,\mu(B,q_0))(t)=\left.\mathscr{L}_R(\dot{\gamma}(0))\right|_{\mu(B,q_0)}
\end{align*}
This concludes the proof of (i).

\item[(ii)] It is a general fact for principal bundles, see \cite{KoNo}.\qedhere
\end{enumerate}
\end{proof}

Additionally, we need the following technical fact regarding the action of sufficiently large subgroups of $SE_0(n,1)$. An analogous result in the Riemannian context can be found in \cite{ChKok}, but the proof is quite different. The main reason is that we need to be very careful with the causal character of the vectors involved in the proof.

\begin{Th}\label{th:action}
Let $G$ be a Lie subgroup of $SE_0(n,1)=\mathbb{R}^{n,1}\ltimes SO_0(n,1)$ such that $\operatorname{pr}_2(G)=SO_0(n,1)$. Then either $G=SE_0(n,1)$ or $G\cong SO_0(n,1)$.

\end{Th}

\begin{proof}
Note that since $\operatorname{pr}_2(G)=SO_0(n,1)$, then for each $A\in SO_0(n,1)$ there exists at least one vector $w_A\in{\mathbb R}^{n,1}$ such that 
\[
\psi_A=(w_A,A)\in G. 
\]

This proof has two main parts, each corresponding to the two cases described for $G$: if $G$ contains a non-trivial translation, then we will show that $G$ is the entire group $SE_0(n,1)$, while on the other hand, if the only translation in $G$ is the identity, then we will see that $G$ is isomorphic to $SO_0(n,1)$.

For the first part, let us suppose there exists a vector $v\in{\mathbb R}^{n,1}$, $v\neq0$, such that $\varphi_v:=(v,\operatorname{id})\in G$. Also, given an arbitrary element $\psi=(w,A)\in G$, we have 
\begin{equation}\label{eq:conj}
\psi^{-1}\circ\varphi_v\circ\psi=\varphi_{A^{-1}v}\in G.
\end{equation}
This conjugation results in another translation in $G$.

Due to technical reasons that will become evident during the proof of the result, the arguments have to take into account the causality of the vector $v$. In each case we will prove that the hypothesis of the existence of one such vector $v\in{\mathbb R}^{n,1}$, $v\neq0$, implies that $\varphi_u\in G$, for all $u\in{\mathbb R}^{n,1}$. Equivalently, if $G$ contains one translation, then it contains all translations. Note that $G$ containing all translations means that $G=SE_0(n,1)$, since 
\[
\varphi_u\circ\psi_A=(u+w_A,A)\in G
\]
for all $u\in{\mathbb R}^{n,1}$, which implies that $(u,A)\in G$ for all $u\in{\mathbb R}^{n,1}$ and all $A\in SO_0(n,1)$.

\begin{description}
\item[Spacelike case $\langle v,v\rangle_J=r^2>0$] Since $SO_0(n,1)$ acts transitively on the Lorentzian sphere
$$S_r(0):=\{u\in\mathbb{R}^{n,1}|\,\langle u,u\rangle_J=r^2\},$$
see \cite{Figue}, then for any $u\in S_r(0)$ there exists an element $A\in SO_0(n,1)$ such that $A^{-1}v=u$. From \eqref{eq:conj}, we conclude that $\psi_A^{-1}\circ\varphi_v\circ\psi_A=\varphi_u\in G$ for any $u\in S_r(0)$.

If $u\in\mathbb{R}^{n,1}$ with $0<\langle u,u\rangle_J<r^2$, then one can decompose 
\[
u=u'+u''\quad\mbox{with}\quad u',u''\in S_r(0). 
\]
The reason is that one can take $u'\in S_r(0)\cap S_r(u)\neq\varnothing$ and then set $u''=u-u'\in S_r(0)$. From the previous argument, we know that $\varphi_{u'},\varphi_{u''}\in G$, and thus
\[
\varphi_{u}=\varphi_{u'}\circ\varphi_{u''}\in G.
\]
We conclude that $\varphi_u\in G$ for any $\langle u,u\rangle_J<r^2$.

Furthermore, if $\langle u,u\rangle_J\,>r^2$, there exists $k\in\mathbb{N}$ such that $\frac1{k^2}\langle u,u\rangle_J=\langle \frac1ku,\frac1ku\rangle_J<r^2$. It follows that $\varphi_{\frac{1}{k}u}\in G$, which implies 
\[
\underbrace{\varphi_{\frac{1}{k}u}\circ\cdots\circ\varphi_{\frac{1}{k}u}}_{k\mbox{ times}}=\varphi_u\in G. 
\]
Thus, we conclude that $\varphi_u\in G$ for every space-like vector $u\in{\mathbb R}^{n,1}$.

To continue with the argument, we need to show that if $u\in{\mathbb R}^{n,1}$ is a timelike vector, then the translations $\varphi_u$ also belong to $G$. Since $SO_0(n,1)$ acts transitively on each of the Lorentzian hyperbolic spaces
\begin{align*}
H_\rho^+&=\{u\in{\mathbb R}^{n,1}|\,\langle u,u\rangle_J=-\rho^2,u_{n+1}>0\},\\
H_\rho^-&=\{u\in{\mathbb R}^{n,1}|\,\langle u,u\rangle_J=-\rho^2,u_{n+1}<0\},
\end{align*}
for any $\rho>0$, see \cite{Figue}, using similar arguments to the spacelike case, it is enough to show that $\varphi_u\in G$, for two specific timelike vectors $u\in{\mathbb R}^{n,1}$, one of them in $H_\rho^+$ and another one in $H_\rho^-$. To do this, consider the spacelike vectors
$$u_+=\left(\begin{array}{c}
\frac{\sqrt{5}}{2}\rho\\
0 \\
\vdots\\
0\\
\frac{1}{2}\rho
\end{array}
\right), u_-=\left(\begin{array}{c}
-\frac{\sqrt{5}}{2}\rho\\
0 \\
\vdots \\
0\\
\frac{1}{2}\rho
\end{array}
\right)\in S_\rho(0).
$$
It is easy to see that
\[
u_++u_-\in H_\rho^+\quad\mbox{and}\quad -u_+-u_-\in H_\rho^-
\]
and therefore 
\[
\varphi_{u_++u_-}=\varphi_{u_+}\circ\varphi_{u_-}\in G\quad\mbox{and}\quad\varphi_{u_++u_-}^{-1}=\varphi_{-u_+-u_-}\in G.
\]
We conclude that for every $u\in H_\rho^+\cup H_\rho^-$, $\varphi_u\in G$. As the choice of $\rho$ is arbitrary, we conclude that $\varphi_u\in G$ for any time-like $u$.

The final step of this case is to consider lightlike vectors. As before, since $SO_0(n,1)$ acts transitively on each half of the light cone
\begin{align*}
L^+&=\{u\in{\mathbb R}^{n,1}|\,\langle u,u\rangle_J=0,u_{n+1}>0\},\\
L^-&=\{u\in{\mathbb R}^{n,1}|\,\langle u,u\rangle_J=0,u_{n+1}<0\},
\end{align*}
see \cite{Figue}, it is enough to show that $\varphi_u\in G$, for two specific lightlike vectors $u\in{\mathbb R}^{n,1}$, one of them in $L^+$ and another one in $L^-$. To do this, consider the vectors
$$u_1=\left(\begin{array}{c}
1\\
0 \\
\vdots\\0\\
0
\end{array}
\right)\in S_1(0)\quad\text{and}\quad u_2=\left(\begin{array}{c}
0 \\0\\
\vdots \\
0 \\
1
\end{array}
\right)\in H_1^+.$$
Since $u_1+u_2\in L^+$ and $u_1-u_2\in L^-$, we see that
\[
\varphi_{u_1+u_2}=\varphi_{u_1}\circ\varphi_{u_2}\in G\quad\mbox{and}\quad\varphi_{u_1-u_2}=\varphi_{u_1}\circ\varphi_{-u_2}\in G.
\]
It follows that $\varphi_u\in G$, for every $u\in L^+\cup L^-$. 

We have shown that if $v\in{\mathbb R}^{n,1}$ is spacelike and $\varphi_v\in G$, then $G$ contains all translations.

\item[Timelike case $\langle v,v\rangle=-r^2<0$] In this case, $v\in H^+_r$ or $v\in H^-_r$. Without loss of generality, let us suppose $v\in H^+_r$, since the other case is treated similarly. Since $SO_0(n,1)$ acts transitively on $H^+_r$, then the translation $\varphi_w\in G$, for all $w\in H^+_r$. As before, that means also that $\varphi_w\in G$, for all $w\in H^-_r$, since $\varphi_w^{-1}=\varphi_{-w}$. Specifically, considering the vectors
\[
w_1=\left(\begin{array}{c}
0 \\
\vdots\\
0\\
\frac{1}{2}r \\
\frac{\sqrt{5}}{2}r
\end{array}
\right)\in H^+_r\quad\text{and}\quad w_2=\left(\begin{array}{c}
0 \\
\vdots \\
0 \\
\frac{1}{2}r \\
-\frac{\sqrt{5}}{2}r
\end{array}
\right)\in H^-_r,
\]
we see that $\varphi_{w_1+w_2}=\varphi_{w_1}\circ\varphi_{w_2}\in G$, where
\[
w_1+w_2=\left(\begin{array}{c}
0 \\
\vdots \\
0 \\
r \\
0
\end{array}
\right)\in S_r(0).
\]
By the first case, since $w_1+w_2\neq0$ is spacelike, we also conclude that $G$ contains all translations.

\item[Lightlike case $\langle v,v\rangle_J=0$] As before, this means that $v\in L^+$ or $v\in L^-$. Without loss of generality, let us suppose $v\in L^+$. By transitivity and by taking inverses, we see that $\varphi_w\in G$, for every $w\in L^+\cup L^-$. Thus, we can take the vectors
$$w_1=\left(\begin{array}{c}
\frac{1}{2}r \\
0\\
\vdots\\
0\\
\frac{1}{2}r
\end{array}
\right)\in L^+\quad\mbox{and}\quad w_2=\left(\begin{array}{c}
\frac{1}{2}r \\
0 \\
\vdots \\
0 \\
-\frac{1}{2}r
\end{array}
\right)\in L^-$$
The vector $w_1+w_2\neq0$ is spacelike, and thus we conclude that $G$ contains all translations as in the previous cases.
\end{description}

To conclude the result, we need to check the second case, that is, if $G$ contains no non-trivial translation, then $G\cong SO_0(n,1)$. The main point in this case is to notice that given $A\in SO_0(n,1)$ there is only one vector $w_A\in{\mathbb R}^{n,1}$ such that $(w_A,A)\in G$. For if $\psi_1=(v_1,A),\psi_2=(v_2,A)\in G$, then
\[
\psi_1\circ\psi_2^{-1}=(v_1,A)\circ(-A^{-1}v_2,A^{-1})=(v_1-v_2,\operatorname{id})\in G
\]
is a translation, therefore $v_1=v_2$ and thus the map
\[
SO_0(n,1)\ni A\mapsto w_A\in{\mathbb R}^{n,1}
\]
is injective. We conclude that, in this case, the surjective map $\operatorname{pr}_2\colon G\to SO_0(n,1)$ is an isomorphism and its inverse is the map $SO_0(n,1)\ni A\mapsto \psi_A\in G$.
\end{proof}

Some of the main results of the next section are based on a direct application of \cite[Theorem 7.1, Chapter IV]{KoNo} to the case of the rolling system studied in \cite{ChKok}. It is immediate to note that Theorem \ref{th:action} above is a generalization to the previously mentioned result of Kobayashi and Nomizu in the Lorentzian context, with one important difference: in the Riemannian case all subgroups $G<SE(n)$ such that $\operatorname{pr}_2(G)=SO(n)$ and $G\neq SE(n)$ have a fixed point. For a given group $G$ one such fixed point can be written down explicitly using the Haar measure and, thus, the compactness of $SO(n)$ plays a crucial role for its existence. 

We have not been able to determine whether all the groups $G\cong SO_0(n,1)$ described in Theorem \ref{th:action} have a fixed point, nevertheless we can show directly that there are plenty of such subgroups of $SE_0(n,1)$. Given $(x_0,A)\in SE_0(n,1)$, let us consider the function
\[
\xi_A(v)=Av-Ax_0+x_0.
\]
It is evident that $x_0$ is indeed fixed by $\xi_A$.

Now, fixing $x_0 \in \mathbb{R}^{n,1}$, consider the map
\begin{equation}\label{dia:action}
\begin{array}{rcl}
\Xi_{x_0} : SO_0(n,1) &\longrightarrow & SE_0(n,1) \\
A&\longmapsto& \xi_A
\end{array}
\end{equation}
The image of $\Xi_{x_0}$ is a subgroup of $SE_0(n,1)$, indeed
\begin{align*}
(\xi_B\circ\xi_A)(v)&=\xi_B(Av-Ax_0+x_0)=B(Av-Ax_0+x_0)-Bx_0+x_0\\
&=BAv-BAx_0+x_0=\xi_{BA}(v),
\end{align*}
and furthermore $\Xi_{x_0}$ is injective, indeed if $A,B\in SO_0(n,1)$ satisfy $\xi_A=\xi_B$, then evaluating these maps on an arbitrary vector $v\in\mathbb{R}^{n,1}$, we see that
\[
Av-Ax_0+x_0=Bv-Bx_0+x_0\implies
(A-B)(v-x_0)=0.
\]
Thus $v-x_0\in\operatorname{ker}(A-B)$, for all $v\in\mathbb{R}^{n,1}$ and therefore $\operatorname{ker}(A-B)=\mathbb{R}^{n,1}$, from which the injectivity of $\Xi_{x_0}$ follows. As an immediate consequence, we have that for each  $x_0 \in \mathbb{R}^{n,1}$, the group
\[
\Xi_{x_0}(SO_0(n,1))<SE_0(n,1)
\]
is isomorphic to $SO_0(n,1)$ and has $x_0$ as a fixed point (one of possibly many). In the Riemannian case all such subgroups arise as in \eqref{dia:action}.








\section{Controllability of the Lorentzian rolling problem}\label{sec:control}

Let $x_0\in M$ and $\gamma\colon I\to M$ an a.c. loop based at $x_0$.
For any initial configuration $q_0=(x_0,\hat{x};A)\in Q$, we can write down the rolling curve $q_{\mathcal{D}_R}(\gamma,q_0)$ starting at $q_0$ and following $\gamma$, that is, satisfying $\pi_{Q,M}\circ q_{\mathcal{D}_R}(\gamma,q_0)=\gamma$, as follows
\begin{equation}\label{eq:pardevel}
q_{\mathcal{D}_R}(\gamma,q_0)(t)=\left(\gamma(t),\hat{x}+A\int_0^1P_s^0(\gamma)\dot{\gamma}(s)\,ds;AP_t^0(\gamma)\right).
\end{equation}
This equality follows from \eqref{eq:rollcurve}, since on $\widehat M={\mathbb R}^{n,1}$ the development map $\widehat\Lambda$ is trivial.

\begin{Th}\label{th:control}
Let $M$ be a development complete Lorentzian $(n+1)-$dimensional manifold, and let $\widehat{M}=\mathbb{R}^{n,1}$.
\begin{enumerate}
\item[(i)] If the rolling problem is completely controllable, then the Lorentzian holonomy group $H$ of $M$ is equal to $SO_0(n,1)$.

\item[(ii)] If the Lorentzian holonomy group of $M$ is $SO_0(n,1)$ and $\mathcal{H}_q<SE_0(n,1)$ contains a nontrivial pure translation, then the rolling problem is completely controllable.


\end{enumerate}
\end{Th}

\begin{proof}
   
\begin{enumerate}
\item[(i)]    First, suppose that the rolling problem is completely controllable. Take $x_0\in M$, $A=\operatorname{id}_{T_{x_0}M}$, and $q_0=(x_0,0;A)\in Q$, where it is understood that $T_0(T_{x_0}M)=T_{x_0}M$.
    
    Let $B\in SO_0(T_{x_0}M)$, and $q=(x_0,0;AB)\in Q$. Since the problem is completely controllable, there exists $\gamma\in\Omega_{x_0}(M)$ such that $q=q_{{\mathcal D}_R}(\gamma,q_0)(1)$. Thus, by equation \eqref{eq:pardevel}, the equality
    $$(x_0,0;AB)=\left(x_0,A\int_{0}^{1}P_s^0(\gamma)\dot{\gamma}(s)ds;AP_1^0(\gamma)\right)$$
    holds, implying that $B=P_1^0(\gamma)\in H$. This proves the necessary condition.

\item[(ii)]
    Now, assume that $H=SO_0(n,1)$. Let $q=(x_0,0;A)\in Q$, and let $\mathcal{H}_q$ be the subgroup of $SE_0(n,1)$ such that $\mu(\mathcal{H}_q\times q)=\pi^{-1}_{Q,M}(x_0)\cap\mathcal{O}_{\mathcal{D}}(q)$, as in Lemma \ref{lem:distr} (ii).
    
    Let us show that $\operatorname{pr}_2(\mathcal{H}_q)=SO_0(n,1)$. Indeed, if $B\in SO_0(n,1)$ and since $A\in \operatorname{Iso}(T_{x_0}M,T_0\widehat{M})$, then we have $A^{-1}BA\in SO_0(T_{x_0}M)$. Therefore, by hypothesis, there exists a loop $\gamma\in\Omega_{x_0}(M)$ such that $A^{-1}BA=P_1^0(\gamma)$. Let $(\hat{y},C)\in\mathcal{H}_q$ be such that $\mu((\hat{y},C)q)=q_{\mathcal{D}}(\gamma,q)(1)$, which from equation \eqref{eq:action} is given by
$$
\mu((\hat{y};C)q)=(x_0,\hat{y};CA)=\left(x_0,A\int_0^1P_s^0(\gamma)\dot{\gamma}(s)\,ds;AP_1^0(\gamma)\right)
=q_{\mathcal{D}}(\gamma,q)(1)$$
then we have
$$
CA=AP_1^0(\gamma)
$$
and therefore $C=AP_1^0(\gamma)A^{-1}=B$, concluding the claim.

 Let $H=H_x$ be the Lorentzian holonomy group of $M$ at $x$. It follows from Lemma 4.2 that if $\mathcal{H}_q$ contains a nontrivial pure translation for some $q=(x_0,0;A)$, then $\mathcal{H}_q=SE_0(n,1)$. In this case, by Lemma 4.1, we have
    $$\pi^{-1}_{Q,M}(x_0)\cap\mathcal{O}_{\mathcal{D}}(q)=\mu(\mathcal{H}_{q}\times q)=\mu(SE_0(n,1)\times q)=\pi^{-1}_{Q,M}(x_0)$$
    and therefore $\mathcal{O}_{\mathcal{D}}(q)=Q$, since $\pi_{\mathcal{O}_{\mathcal{D}}(q),M}$ is a submanifold of $\pi_{Q,M}$. Thus, the rolling problem is completely controllable.\qedhere
\end{enumerate}
\end{proof}

It is relevant to contrast the previous result with \cite[Theorem 4.3]{ChKok}. Besides the fact that the lack of compactness of $SO_0(n,1)$ implies that we cannot assert the existence of fixed points for subgroups of $SE_0(n,1)$ as mentioned after the proof of Theorem \ref{th:action}, there is another important technical difference which is not available in the Lorentzian context: in the Riemannian case it can be shown that the exponential is a local isometry, thus turning the analogous result into an equivalence. Many computations in the proof of \cite[Theorem 4.3]{ChKok} can be repeated in the Lorentzian context assuming the action of $\mathcal{H}_q$ has a fixed point, but the same conclusion cannot be reached (at least not by these methods).

Note that the first part of the previous proof allows for an immediate generalization, with exactly the same proof.

\begin{Pro}\label{prop:anysign}
Let $M$ be a development complete pseudo-Riemannian $(n+\nu)-$dimensional manifold of index $\nu$, and let $\widehat{M}=\mathbb{R}^{n,\nu}$. If the rolling problem is completely controllable, then the pseudo-Riemannian holonomy group $H$ of $M$ is equal to $SO_0(n,\nu)$.
\end{Pro}

\begin{Ob}
It is not immediate to generalize the second part of Theorem \ref{th:control} to arbitrary index, since its proof depends on Theorem \ref{th:action} which uses the fact that the space is Lorentzian.
\end{Ob}

\section{Examples}\label{sec:ex}

It is known that full Lorentzian holonomy is the generic situation in Lorentzian geometry, see \cite{GuMu}. The most extensive and careful study of Lorentzian holonomy is the seminal paper by Leistner \cite{Leistner}. As a consequence, plenty of completely controllable examples can be constructed making sure that the group ${\mathcal H}_q<SE_0(n,1)$ contains a nontrivial pure translation. The following examples are more interesting.

The careful construction of rolling maps presented in \cite{MSL}, extending the case of the Lorentzian spheres shown in \cite{AnnFa}, implies that the $(n+\nu)-$dimensional hyperquadrics of index $\nu$
\[
H^{n,\nu}(r)=\{p\in{\mathbb R}^{n,\nu+1}\colon\langle p,p\rangle_J=-r^2\}
\]
and
\[
S^{n,\nu}(r)=\{p\in{\mathbb R}^{n+1,\nu}\colon\langle p,p\rangle_J=r^2\}
\]
are development complete, for any $r>0$. Since the rolling systems of both $H^{n,\nu}(r)$ and $S^{n,\nu}(r)$ on their respective tangent spaces are shown to be controllable by \cite[Theorem 5.3]{MSL}, as a consequence of Proposition \ref{prop:anysign} it is possible to conclude that:
\begin{Cor}
For any $r>0$, the pseudo-Riemannian holonomy of the family of hyperquadrics $H^{n,\nu}(r)$ and $S^{n,\nu}(r)$ is $SO_0(n,\nu)$.
\end{Cor}

\section{Acknowledgments} 
The results of this paper are part of the Ph.D. thesis of the first author at Universidad de La Frontera, Temuco, Chile.

\end{document}